\documentclass[11pt]{article}
 \usepackage{multirow}
\usepackage{amsmath,amssymb,amsthm}
\usepackage{mathtools}
\usepackage{array}
\usepackage{booktabs}
\usepackage{color}
\usepackage{cases}
\usepackage{todonotes}
\usepackage{graphicx}
\usepackage{caption}
\usepackage{subcaption}
\usepackage{comment}
\usepackage{todonotes}

\newtheorem{theorem}{Theorem}[section]
\newtheorem{lemma}[theorem]{Lemma}
\newtheorem{remark}[theorem]{Remark}

\usepackage{placeins}

\usepackage{empheq}

\usepackage[normalem]{ulem}

\newcommand{\vertiii}[1]{{\vert\kern-0.25ex\vert\kern-0.25ex\vert #1 
    \vert\kern-0.25ex\vert\kern-0.25ex\vert}}

\newcommand{\norm}[1]{\lVert#1\rVert}
\newcommand{\seminorm}[1]{| #1 |}
\newcommand{\snorm}[1]{\vertiii{ #1 }_{*}}
\newcommand{\bnorm}[1]{\vertiii{ #1 }_{\mathcal{B}_h}}

\title{Discontinuous Galerkin approximation of linear parabolic problems with  dynamic boundary conditions}
\author{
P.F. Antonietti \thanks{MOX-Dipartimento di Matematica, Politecnico di Milano, P.zza Leonardo Da Vinci 32, I-20133 Milano, Italy (paola.antonietti@polimi.it).}
\and M. Grasselli \thanks{Dipartimento di Matematica, Politecnico di Milano, P.zza Leonardo Da Vinci 32, I-20133 Milano, Italy (maurizio.grasselli@polimi.it).}
\and S. Stangalino\thanks{MOX-Dipartimento di Matematica, Politecnico di Milano, P.zza Leonardo Da Vinci 32, I-20133 Milano, Italy (simone.stangalino@polimi.it).}
\and M. Verani\thanks{MOX-Dipartimento di Matematica, Politecnico di Milano, P.zza Leonardo Da Vinci 32, I-20133 Milano, Italy (marco.verani@polimi.it).}}

\begin{document}
\maketitle

\begin{abstract}
In this paper we propose and analyze a Discontinuous Galerkin method for a 
linear parabolic problem with dynamic boundary conditions. We present the formulation and prove stability 
and optimal a priori error estimates for the fully discrete scheme. More precisely, using polynomials of degree $p\geq 1$ on meshes with granularity $h$ along with a backward Euler time-stepping scheme with time-step $\Delta t$, we prove that the fully-discrete solution is bounded by the data
and it
converges, in a suitable (mesh-dependent) energy norm, to the exact solution with optimal order $h^p + \Delta t$. The sharpness of the theoretical estimates are verified through several numerical experiments. 
\end{abstract}

\section{Introduction}

In this paper we present and analyze a Discontinuous Galerkin (DG) method for the following linear parabolic problem  supplemented with dynamic boundary conditions on $\Gamma_1$:
\begin{equation}
\left\{\label{pb:intro}\begin{array}{lll}
\partial_t u &= \Delta u + f, &\text{ in }\Omega,\ 0<t\leq T,\\
\partial_n u &= -\alpha u + \beta \Delta_{\Gamma}u - \lambda \partial_t u + g,&\text{ on }\Gamma_1,\ 0<t\leq T,\\
\multicolumn{2}{l}{\text{periodic boundary conditions},}&\text{ on }\Gamma_2,\ 0<t\leq T,\\
u_{|t=0}&=u_0, & \text{ in } \overline{\Omega}.
\end{array}\right.
\end{equation}
Here the domain $\Omega$ and the subsets $\Gamma_i\subset \partial \Omega$, $i=1,2$, are depicted in Figure~\ref{identification},
$\Delta_{\Gamma}$ is the Laplace-Beltrami operator, $\partial_n u$ denotes the outer normal derivative of
$u$ on $\Gamma_1$, $g$  is a given function and $\alpha, \beta, \lambda$ are suitable non-negative constants.  
%

Dynamic boundary conditions have been recently considered by physicists to model the fluid interactions with the domain's walls (see, e.g., \cite{dyn1,dyn2,dyn3}). Despite the practical relevance of this kind of boundary conditions from a modeling point of view and the intense research activity  to understand their analytical properties, see, e.g., \cite{Hintermann:1989, Vazquez-Vitillaro:2009,Vazquez-Vitillaro:2011}, the study of suitable numerical methods for their discretization is still in its infancy. To the best of our knowledge, the only work along this direction is \cite{PierreCHDYN}, where the authors analyze a conforming finite element method for the approximation of the Cahn-Hilliard equation supplemented with dynamic boundary conditions. 
Motivated by the flexibility and versatility of DG methods, here we propose and  analyze a DG method combined with a backward Euler time advancing scheme
for the discretization of a linear parabolic problem with dynamic boundary conditions. 
The main goal of the present work is the numerical treatment of dynamic boundary conditions within the DG framework. Here we consider just a linear equation. However, our results aim to be a key step towards the extension to (non-linear) partial differential equations with dynamic boundary conditions, as, for example, the Cahn-Hilliard equation. 
In this context, we mention  DG methods have been already proved to be an effective discretization strategy for the Cahn-Hilliard equation as shown in \cite{Kay-Styles-Suli:2009} where the authors constructed and analyzed a DG method coupled with a backward Euler time-stepping scheme for a Cahn-Hilliard equation in two-dimensions, cf. also \cite{Wells-Kuhl-Garikipati:2006}.\\

The origins of DG methods can be backtracked to \cite{ReedHill1973, LasaintRaviart_1974} where they have introduced for the discretization of the neutron transport equation. Since that time, DG methods for the numerical solution of partial differential equations have enjoyed a great development, see the monographs    \cite{riviere,HesthavenWarburton_2008} for an overview, and \cite{arnold2002unified} for a unified analysis of DG methods for elliptic problems.
In the context of parabolic equations,
DG methods in primal form combined with backward Euler and Crank-Nicholson time advancing techniques have been firstly analyzed in \cite{arnold1982interior,Riviere_Wheeler_2000}, respectively. DG in time methods have also been studied for parabolic partial differential equations, see, for example, \cite{Jamet_1978,ErikssonJohnson_1991,ErikssonJohnson_1995,LarssonThomeeWahlbin_1998} and the reference therein; cf. also \cite{SchtzauSchwab_2000a,SchtzauSchwab_2000b} for the $hp$-version of the DG time-stepping method.\\

The paper is organized as follows. In Section~\ref{s:basic} we introduce some useful notation and the functional setting. Section~\ref{s:stationary} is devoted to the introduction and analysis of a DG method for a suitable auxiliary (stationary) problem. These results will be then employed in Section~\ref{s:parabolic} to design a DG scheme to approximate the linear parabolic problem with boundary conditions and to obtain optimal a priori error estimates for the fully discrete scheme. Finally, in Section~\ref{s:num} we numerically assess the validity of our theoretical analysis.

\section{Notation and functional setting}\label{s:basic}
In this section we introduce some notation and the functional setting.\\

Let $D\subset \mathbb{R}^2$ be an open, bounded, polygonal domain with boundary $\Gamma=\partial D$.
On $D$ we define the standard Sobolev space $H^s(D),\ s=0,1,2,\ldots$ (for $s=0$ we write $L^2(D)$ instead of $H^0(D)$) and endow it with the usual inner scalar product $(\cdot,\cdot)_{H^s(D)}$, 
and its induced norm $\norm{\cdot}_{H^s(D)}$, cf. \cite{Adams}.  We also need the seminorm defined by 
$\seminorm{\cdot}_{H^s(D)}= (\sum_{\vert \alpha \vert = s} \| \partial^\alpha (\cdot)\|_{L^2(\Omega)})^{1/2}$.\\
We next introduce, on $\Gamma$,  
the Laplace-Beltrami operator. We first define the projection matrix $\textbf{P}=\textbf{I}-\textbf{n} \otimes \textbf{n}=(\delta_{ij}-n_in_j)_{i,j=1}^{2}$, where $\textbf{n}$ is the outward unit normal to $D$, $\textbf{a}\otimes \textbf{b}=(a_ib_j)_{ij}$ is the dyadic product, and $\delta_{ij}$ is the Kroneker delta. We define the tangential gradient of a (regular enough) scalar function $u:\Gamma \rightarrow \mathbb{R}$ as $\nabla_\Gamma u = \textbf{P}\nabla u$.
The tangential divergence of a vector-valued function $\textbf{A}:\Gamma\rightarrow \mathbb{R}^2$ is defined as $\text{div}_\Gamma (\textbf{A})=\text{Tr}\big( (\nabla \textbf{A}) \textbf{P}\big)$, being Tr$(\cdot)$ the trace operator. With the above notation, we define the Laplace-Beltrami operator as $\Delta_\Gamma u = \text{div}_\Gamma(\nabla_\Gamma u)$.\\

We next introduce the following Sobolev surface space
\begin{equation*}
\begin{aligned}
& H^s(\Gamma)=\{v \in H^{s-1}(\Gamma)\ |\ \nabla_\Gamma v \in [H^{s-1}(\Gamma)]^2\}, 
&& s\geq 1,
\end{aligned}
\end{equation*}
cf. \cite{dziuk1988finite}, with the convention that $H^0(\Gamma)\equiv L^2(\Gamma)$, $L^2(\Gamma)$ being the standard Sobolev space of square integrable functions (equipped with the usual inner scalar product $(\cdot,\cdot)_{\Gamma}$ and the usual induced norm $\norm{\cdot}_{L^2(\Gamma)}$).
We equipped the space $H^s(\Gamma)$ with the following surface seminorm
and norm
\begin{equation*}
\begin{aligned}
\seminorm{v}_{H^s(\Gamma)}&= \norm{\nabla_\Gamma v}_{H^{s-1}(\Gamma)}  
&& \forall v \in H^s(\Gamma),\ s\geq 1,\\
\norm{v}_{H^s(\Gamma)}&=\sqrt{\norm{v}^2_{H^{s-1}(\Gamma)}+\seminorm{v}^2_{H^s(\Gamma)}} 
&& \forall v \in H^s(\Gamma),\ s\geq 1,\\
\end{aligned}
\end{equation*}
respectively.
In \cite[Lemma 2.4]{Kashiwabara2014} is proved that the above norm is equivalent to the usual surface norm present in literature \cite{MR0350177}, which is defined in local coordinates after a truncation by a partition of unity. \\

Next, for a positive constant $\lambda$, we introduce the space 
\begin{equation*}
\begin{aligned}
H^s_\lambda(D,\Gamma)=\{v\in H^s(D) : \lambda v_{|\Gamma}\in H^s(\Gamma)\},\qquad s\geq 0,
\end{aligned}
\end{equation*}
 and endow it with the norm
\begin{equation*}\norm{u}_{H^s_\lambda(D,\Gamma)}=\sqrt{\left( \norm{u}^2_{H^s(D)} + \lambda \norm{u_{|\Gamma}}^2_{H^s(\Gamma)}\right)}.
\end{equation*}
As before, for $s=0$ we will write $H^s_\lambda(D,\Gamma)$ instead of $H^0_\lambda(D,\Gamma)$. Moreover, to ease the notation, when $\lambda=1$, we will omit the subscript.\\

Finally, throughout the paper, we will write $x \lesssim y$ to signify $x\leq Cy$, where $C$ is a generic positive constant whose value, possibly different at any occurrence, does not depend on the discretization parameters.

\section{The stationary problem and its DG discretization}\label{s:stationary}
Let $\Omega=(a,b)\times(c,d)\subset \mathbb{R}^2$ be a rectangular domain and let $\Gamma_1, \Gamma_2$ be the union of the top and bottom/left and right edges, respectively, cf. Figure~\ref{identification}. 
We consider the following Laplace problem with generalized Robin boundary conditions:
\begin{equation}
\label{mainpb}\left\{ \begin{array}{rll}
-\Delta u &=f, &\text{ in }\Omega,\\
\partial_n u &= - \alpha u + \beta \Delta_{\Gamma}u+g,&\text{ on }\Gamma_1, \\
\multicolumn{2}{l}{\text{periodic boundary conditions,}}&\text{ on }\Gamma_2,\\
\end{array}\right.
\end{equation}
where $\alpha, \beta$ are positive constants, and $f\in {L^2(\Omega)}$, $g\in L^2(\Gamma_1)$ are given functions.

Defining the bilinear form $a(u,v):H^1(\Omega,\Gamma_1) \times H^1(\Omega,\Gamma_1) \to \mathbb{R}$ as $$a(u,v)= (\nabla u, \nabla v)_{L^2(\Omega)} + \beta (\nabla_\Gamma u, \nabla_\Gamma v)_{L^2(\Gamma_1)} + \alpha 
(u, v)_{L^2(\Gamma_1)},$$
the weak formulation of (\ref{mainpb}) reads:  find $u \in H^1(\Omega,\Gamma_1)$ such that
\begin{equation}\label{weakpb}
a(u,v)= ( f,v )_{L^2(\Omega)}+(g,v)_{L^2(\Gamma_1)} \qquad \forall v \in H^1(\Omega,\Gamma_1).
\end{equation}
The following result shows that formulation (\ref{weakpb}) is well posed.
\begin{theorem}\label{th_principe}
Problem (\ref{weakpb}) admits a unique solution $u\in H^2(\Omega,\Gamma_1)$ satisfying the following stability bound:
\begin{equation}\label{estimate_stab}
\norm{u}_{H^2(\Omega,\Gamma_1)} \lesssim \norm{f}_{L^2(\Omega)} + \norm{g}_{L^2(\Gamma_1)}.
\end{equation}
Moreover, if $f\in H^{s-2}(\Omega)$ and $g\in H^{s-2}(\Gamma_1)$, $s\geq 2$, then $u\in H^s(\Omega,\Gamma_1)$ and
\begin{equation}\label{estimate_stab2}
\norm{u}_{H^s(\Omega,\Gamma_1)} \lesssim \norm{f}_{H^{s-2}(\Omega)} + \norm{g}_{H^{s-2}(\Gamma_1)}.
\end{equation}
\end{theorem}
\begin{proof}
The existence and uniqueness of the solution are proved in \cite[Theorem 3.2]{Kashiwabara2014}.
The proof of the regularity results is shown in \cite[Theorem 3.3-3.4]{Kashiwabara2014}. The same arguments used in \cite[Theorem 3.3-3.4]{Kashiwabara2014} apply also in our case thanks to periodic conditions.
\end{proof}

\begin{remark}\label{smoothassumption}
We observe that the forthcoming analysis holds in more general-shaped domains and/or more general type of boundary conditions provided that the exact solution of the differential problem analogous to (\ref{weakpb}) satisfies a stability bound of the form of (\ref{estimate_stab}).
\end{remark}

\subsection{Discontinuous Galerkin space discretization}\label{dgdiscr}
In this Section~we present a discontinuous Galerkin (DG) approximation of problem (\ref{weakpb}).\\
Let $\mathcal{T}_h$ be a quasi-uniform partition of $\Omega$ into disjoint open triangles $T$ such that $\overline{\Omega}=\cup_{T\in\mathcal{T}_h} \overline{T}$. We set $h=\max\{\text{diam}(T),\ T \in \mathcal{T}_h\}$.
For $s\geq 0$, we define the following broken space $$H^s(\mathcal{T}_h)=\{v\in L^2(\Omega) : v_{|T} \in H^s(T,\partial T),\ T\in\mathcal{T}_h \},$$ where, as before, $H^0(\mathcal{T}_h)=L^2(\mathcal{T}_h)$. For an integer $p\geq 1$, we also define the finite dimensional space $$V^p(\mathcal{T}_h)=\{v\in L^2(\Omega) : v_{|T} \in \mathbb{P}^p(T),\ T\in\mathcal{T}_h \}\subset H^s(\mathcal{T}_h),$$
for any $s \geq 0$.
An interior edge $e$ is defined as the non-empty intersection of the closure of two neighboring elements, i.e., $\overline{e}=\overline{T_1}\cap\overline{T_2}$, for $T_1, T_2 \in \mathcal{T}_h$. We collect all the interior edges in the set $\mathcal{E}_h^0$.
Recalling that on $\Gamma_2\subset \partial\Omega$ we impose periodic boundary conditions, we decompose $\Gamma_2$ as $\Gamma_2=\Gamma_2^+\cup \Gamma_2^-$, cf. Figure~\ref{identification} (left), and identify $\Gamma_2^+$ with $\Gamma_2^-$, cf. Figure~\ref{identification} (right). Then we define the set $\mathcal{E}_h^{\Gamma_2}$ of the periodic boundary edges as follows. An edge $e\in\mathcal{E}_h^{\Gamma_2}$ if $\overline{e}=\partial \overline{T}^- \cap \partial \overline{T}^+$, where $T^{\pm}\in \mathcal{T}_h$ such that $\partial T^{\pm}\subseteq \Gamma_2^{\pm}$, cf. Figure~\ref{identification} (right).
We also define a boundary edge $e_{\Gamma_1}$ as the non-empty intersection between the closure of an element in $\mathcal{T}_h$ and $\Gamma_1$ and the set of those edges by $\mathcal{E}_h^{\Gamma_1}$.
Finally, we define a boundary ridge $r$ as the subset of the mesh vertexes that lie on $\Gamma_1$, and collect all the ridges $r$ in the set $\mathcal{R}_h^{\Gamma_1}$. Clearly, the corner ridges have to be identified according to the periodic boundary conditions (cf. Figure~\ref{identification}, right).
The set of all edge will be denoted by $\mathcal{E}_h$, i.e., $\mathcal{E}_h=\mathcal{E}_h^0\cup \mathcal{E}_h^{\Gamma_1}\cup \mathcal{E}_h^{\Gamma_2}$.
\begin{figure}[!hbt]
\centering
\begin{subfigure}[b]{0.35\textwidth}                \includegraphics[height=0.8\textwidth]{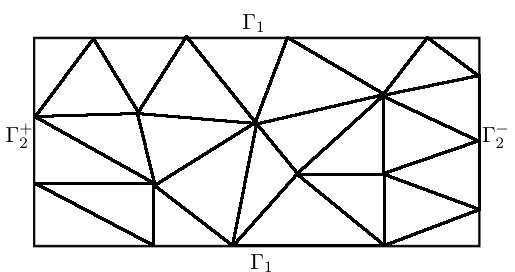}
\end{subfigure}\qquad \qquad \qquad \qquad
\begin{subfigure}[b]{0.25\textwidth}                \includegraphics[height=1.12\textwidth]{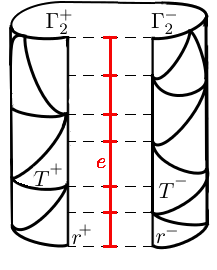}
\end{subfigure}
        \caption{Example of a domain $\Omega$ and an admissible triangulation $\mathcal{T}_h$ (left). On the right, we highlight the edges $e\in \mathcal{E}_h^{\Gamma_2}$ with red lines.}\label{identification}
\end{figure}

For $v\in H^s(\mathcal{T}_h)$, $s\geq 1$, we define
$$\seminorm{v}_{H^s(\mathcal{T}_h)}^2= \sum_{T \in \mathcal{T}_h} \seminorm{v}_{H^s(T)}^2, \qquad
\seminorm{v}_{H^s(\mathcal{E}_h^{\Gamma_1})}^2= \sum_{e_{\Gamma_1} \in \mathcal{E}_h^{\Gamma_1}} \seminorm{v}_{H^s(e_{\Gamma_1})}^2.$$
Next, for each $e\in \mathcal{E}_{h}^0 \cup \mathcal{E}_{h}^{\Gamma_2}$ we define the jumps and the averages of $v\in H^1(\mathcal{T}_h)$ as
$$[v]_e=(v^+) \textbf{n}_{e}^+ + (v^-) \textbf{n}_{e}^- \qquad \text{and} \qquad \{v\}_{e} = \frac{1}{2}(v^++v^-),$$
where $v^\pm=v_{|T^\pm}$ and $\textbf{n}_{e}^\pm$ is the unit normal vector to $e$ pointing outward of $T^\pm$. For each $e\in\mathcal{E}_h^{\Gamma_1}$ we define
$$[v]_{e}=v_{|e}\ \mathbf{n}_{e},\qquad \{v\}_{e}=v_{|e}, \qquad v\in H^1(\mathcal{T}_h).$$
Analogously, for each $r \in \mathcal{R}_h^{\Gamma_1}$, we set
$$[v]_{r}=(v^+(r)) \textbf{n}_{r}^+ + (v^-(r)) \textbf{n}_{r}^- \qquad \text{and} \qquad \{v\}_{r} = \frac{1}{2}(v^+(r)+v^-(r)),$$
where, denoting by $e^\pm$ the two edges sharing the ridge $r$, $v^\pm(r)=v_{|{e^\pm}}(r)$ and $\textbf{n}_{r}^\pm$ is the unit tangent vector to $\Gamma_1$ on $r$ pointing outward of ${e^\pm}$.
The above definitions can be immediately extended to a (regular enough) vector-valued function, cf. \cite{arnold2002unified}.
To simplify the notation, when the meaning will be clear from the context,
we remove the subscripts from the jump and average operators. Adopting the convention that
$$(v,w)_{\mathcal{E}_{h}}=\sum_{e \in \mathcal{E}_{h}} (v,w)_{L^2(e)}, \qquad (\xi,\eta)_{\mathcal{R}_{h}^{\Gamma_1}}=\sum_{r \in \mathcal{R}_{h}^{\Gamma_1}} \xi(r)\eta(r)$$
for regular enough functions $v,w,\xi,\eta$, we introduce the following bilinear forms
\begin{align*}
\mathcal{B}_{h}(v,w)=\sum_{T \in \mathcal{T}_h} (\nabla v ,\nabla w)_T &- ([v],\{ \nabla w \})_{\mathcal{E}_{h}^0} - ([w], \{\nabla v\})_{\mathcal{E}_{h}^0} + \sigma ([v],[w])_{\mathcal{E}_{h}^0}\\
& -  ([v],\{ \nabla w \})_{\mathcal{E}_{h}^{\Gamma_2}} - ([w], \{\nabla v\})_{\mathcal{E}_{h}^{\Gamma_2}} + \sigma ([v],[w])_{\mathcal{E}_{h}^{\Gamma_2}}
\end{align*}
and
$$b_h(v,w)= (\nabla_\Gamma v, \nabla_\Gamma w)_{\mathcal{E}_h^{\Gamma_1}} -( [v],\{ \nabla_\Gamma w \})_{\mathcal{R}_{h}^{\Gamma_1}} - ([w],\{\nabla_\Gamma v\})_{\mathcal{R}_{h}^{\Gamma_1}} + \sigma([v],[w])_{\mathcal{R}_{h}^{\Gamma_1}},$$
for all $v,w \in H^2(\mathcal{T}_h)$. Here $\sigma=\frac{\gamma}{h}$, being $\gamma$ a  positive constant at our disposal.
We then set 
\begin{equation}\label{mcalah}
\mathcal{A}_h(u,v)=\mathcal{B}_{h}(u,v)+\alpha\  (u,v)_{L^2(\Gamma_1)} +\beta\ b_{h}(u,v).
\end{equation}

The discontinuous Galerkin approximation of problem (\ref{mainpb}) reads: find $u_h\in V^p(\mathcal{T}_h)$ such that
\begin{equation}\label{dgpb}
\mathcal{A}_h(u_h,v_h)= ( f,v_h )_{L^2(\Omega)}+(g,v_h)_{L^2(\Gamma_1)} \qquad \forall v_h \in V^p(\mathcal{T}_h).
\end{equation}
In the following we show that the bilinear form $\mathcal{A}_h(\cdot,\cdot)$ is continuous and coercive in a suitable (mesh-dependent) energy norm. To this aim, for $w\in H^s(\mathcal{T}_h)$, we define the seminorm
\begin{equation*}
\bnorm{w}^2=\seminorm{w}_{H^1(\mathcal{T}_h)}^{2}+\sigma \norm{[w]}_{L^2(\mathcal{E}_h^0\cup \mathcal{E}_h^{\Gamma_2})}^{2} + \frac{1}{\sigma}\norm{\{\nabla w\}}_{L^2(\mathcal{E}_h^0\cup \mathcal{E}_h^{\Gamma_2})}^{2}
\end{equation*}
and the norm 
\begin{multline}\label{norm_star}
\snorm{w}^2= \bnorm{w}^2+\alpha \norm{w}_{L^2(\Gamma_1)}^{2}\\
+\beta \seminorm{w}_{H^1(\mathcal{E}_h^{\Gamma_1})}^{2}+
\beta\sigma \norm{[w]}_{L^2(\mathcal{R}_h^{\Gamma_1})}^{2} + \frac{\beta}{\sigma}\norm{\{ \nabla_\Gamma w\}}_{L^2(\mathcal{R}_h^{\Gamma_1})}^{2},
\end{multline}
where we adopted the notation $$\norm{w}^2_{L^2(\mathcal{E}_{h})}=\sum_{e \in \mathcal{E}_{h}} \norm{w}^2_{L^2(e)}, \qquad  \norm{w}^2_{L^2(\mathcal{R}_{h}^{\Gamma_1})}=\sum_{e \in \mathcal{R}_{h}^{\Gamma_1}} \norm{w}^2_{L^2(r)}.$$
Reasoning as in \cite{arnold1982interior}, it is easy to prove the following result.
\begin{lemma}
It holds
\begin{equation}\label{continuity_a}
\mathcal{A}_h(v,w)\lesssim \snorm{v}\snorm{w} \qquad \forall v,w\in H^2(\mathcal{T}_h).
\end{equation}
Moreover, for $\gamma$ large enough, it holds
\begin{equation}\label{coercivity_a}
\snorm{v}^{2}\lesssim \mathcal{A}_h(v,v) \qquad \forall v\in V^p(\mathcal{T}_h).
\end{equation}
\end{lemma}
\begin{proof}
Let us first prove \eqref{continuity_a}. The term $\mathcal{B}_{h}(\cdot,\cdot)$ can be bounded by Cauchy-Schwarz inequality as in \cite{arnold1982interior}. Also the term $b_{h}(\cdot,\cdot)$ can be handled using the Cauchy-Schwarz inequality:
\begin{align*}
| b_{h}(v,w) |&= \left|(\nabla_\Gamma v, \nabla_\Gamma w)_{\mathcal{E}_h^{\Gamma_1}}  - ( [v],\{ \nabla_\Gamma w \})_{\mathcal{R}_h^{\Gamma_1}}\right.\\
& \left.\qquad \qquad \qquad \qquad \qquad \qquad - ([w],\{\nabla_\Gamma v\})_{\mathcal{R}_h^{\Gamma_1}} + \sigma ([v],[w])_{\mathcal{R}_h^{\Gamma_1}}\right|\\
&\lesssim \Big( \seminorm{v}_{H^1(\mathcal{E}_h^{\Gamma_1})}^{2}+\sigma \norm{[v]}_{L^2(\mathcal{R}_h^{\Gamma_1})}^{2} + \frac{1}{\sigma}\norm{\{ \nabla_\Gamma v\}}_{L^2(\mathcal{R}_h^{\Gamma_1})}^{2}\Big)^{1/2}\times \\
&\qquad \qquad \Big( \seminorm{w}_{H^1(\mathcal{E}_h^{\Gamma_1})}^{2}+\sigma \norm{[w]}_{L^2(\mathcal{R}_h^{\Gamma_1})}^{2} + \frac{1}{\sigma}\norm{\{ \nabla_\Gamma w\}}_{L^2(\mathcal{R}_h^{\Gamma_1})}^{2}\Big)^{1/2},
\end{align*}
and (\ref{continuity_a}) follows employing the definition (\ref{norm_star}) of the norm $\snorm{\cdot}$.\\
We now prove \eqref{coercivity_a}. As before the term $\mathcal{B}_{h}(\cdot,\cdot)$ can be bounded as in \cite{arnold1982interior}: using the classical polynomial inverse inequality \cite{Ciarlet2002} we obtain
$$\bnorm{v}^2\lesssim \seminorm{v}_{H^1(\mathcal{T}_h)}^2 + \sigma \norm{[v]}^2_{L^2(\mathcal{E}^0_h\cup \mathcal{E}^{\Gamma_2}_h)}\lesssim \mathcal{B}_{h}(v,v)$$
for all $v\in V^p(\mathcal{T}_h).$
The term $b_{h}(\cdot,\cdot)$ can be estimated as follows:
\begin{equation*}
b_{h}(v,v)\geq \seminorm{v}_{H^1(\mathcal{E}_h^{\Gamma_1})}^2 - 2\left|( [v],\{ \nabla_\Gamma v \})_{\mathcal{R}_h^{\Gamma_1}}\right| + \sigma \norm{[v]}^2_{L^2(\mathcal{R}_h^{\Gamma_1})}.
\end{equation*}
Employing the arithmetic-geometric inequality we get:
\begin{align*}
\left|( [v],\{ \nabla_\Gamma v \})_{\mathcal{R}_h^{\Gamma_1}}\right| &\leq \norm{\sigma^{1/2}[v]}_{L^2(\mathcal{R}_h^{\Gamma_1})}\norm{\{ \sigma^{-1/2} \nabla_\Gamma v \}}_{L^2(\mathcal{R}_h^{\Gamma_1})}\\
&\leq \frac{1}{\epsilon} \sigma\norm{[v]}_{L^2(\mathcal{R}_h^{\Gamma_1})}^2 + 4 \epsilon \sigma^{-1}\norm{\{ \nabla_\Gamma v \}}_{L^2(\mathcal{R}_h^{\Gamma_1})}^2,
\end{align*}
for a positive $\epsilon>0$. Finally, estimate (\ref{coercivity_a}) follows using the polynomial inverse inequality
$$h \norm{ \{ \nabla_\Gamma v \}}_{L^2(\mathcal{R}_h^{\Gamma_1})}^2\lesssim \seminorm{v}_{H^1(\mathcal{E}_h^{\Gamma_1})}^2 \qquad \forall v\in V^p(\mathcal{T}_h)$$
and choosing $\gamma$ sufficiently large.
\end{proof}

The following result shows that problem (\ref{dgpb}) admits a unique solution and that the Galerkin orthogonality property is satisfied. The proof is straightforward and we omit it for sake of brevity.

\begin{lemma} Assume that $\gamma$ is sufficiently large.
Then, the discrete solution $u_h$ of problem (\ref{dgpb}) exists and is unique. Moreover, formulation (\ref{dgpb}) is strongly consistent, i.e., 
\begin{equation}\label{Go}
\mathcal{A}_h(u-u_h,v)=0\qquad \forall v\in V^p(\mathcal{T}_h).
\end{equation}
\end{lemma}

For $v\in H^{s}(\Omega,\Gamma_1)$, $s\geq 2$, let $I^h_p v$ be the piecewise Lagrangian interpolant of order $p$ of $u$ on $\mathcal{T}_h$. Note that $(I^h_p u)_{|\Gamma_1}$ interpolates $u$ on the set of degrees of freedom that lie on $\mathcal{E}_h^{\Gamma_1}$.
By standard approximation results we get the following interpolation estimate.

\begin{lemma}\label{S_IntEstimates}
For all $v \in H^s(\Omega,\Gamma_1)$, $s\geq 2$, it holds
$$\snorm{v-I^h_p v} \lesssim h^{\min{(s-1,p)}} \norm{v}_{H^s(\Omega,\Gamma_1)}.$$
\end{lemma}
\begin{proof}
Using the definition (\ref{norm_star}) of $\snorm{\cdot}$ norm and that $I_p^h v(r) = v(r)$ for all $r \in \mathcal{R}_h^{\Gamma_1}$, we get
\begin{equation}\label{*norm}
\snorm{v-I^h_p v}^2= \bnorm{v-I^h_p v}^2+\alpha \norm{v-I^h_p v}_{L^2(\Gamma_1)}^{2}+\beta \seminorm{v-I^h_p v}_{H^1(\mathcal{E}_h^{\Gamma_1})}^{2}.
\end{equation}
Expanding the first term at right-hand side and using the multiplicative trace inequalities
\begin{align*}
\norm{v}^2_{L^2(\mathcal{E}_{h})} \lesssim h^{-1} \norm{v}^2_{L^2(\Omega)} + h \seminorm{v}_{H^1(\Omega)}^2,\\
\norm{\nabla v}^2_{L^2(\mathcal{E}_{h})} \lesssim h^{-1} \seminorm{v}^2_{H^1(\Omega)} + h \seminorm{v}_{H^2(\Omega)}^2,
\end{align*}
cf. \cite{riviere}, we get
\begin{align*}
\bnorm{v-I^h_p v}^2&=\seminorm{v-I^h_p v}_{H^1(\Omega)}^{2}+\sigma \norm{[v-I^h_p v]}_{L^2(\mathcal{E}_{h}^0\cup\mathcal{E}_{h}^{\Gamma_2})}^{2} \\
& \qquad\qquad\qquad\qquad\qquad\qquad+ \frac{1}{\sigma}\norm{\{\nabla (v-I^h_p v)\}}_{L^2(\mathcal{E}_{h}^0\cup\mathcal{E}_{h}^{\Gamma_2})}^{2}\\
&\lesssim h^{-2} \norm{v-I^h_p v}^2_{L^2(\Omega)} + \seminorm{v-I^h_p v}_{H^1(\Omega)}^{2}
+ h^2 \seminorm{v-I^h_p v}_{H^2(\Omega)}^2.
\end{align*}
Using standard interpolation estimates \cite{PerugiaSchotzau} we get the thesis.
\end{proof}

Now we show that the discrete solution $u_h$ of (\ref{dgpb}) converges to the weak solution of (\ref{weakpb}).

\begin{theorem}\label{StatTh}
Let $u\in H^s(\Omega,\Gamma_1)$, $s\geq 2$, be the solution of the problem (\ref{weakpb}) and let $u_h$ be the solution of the problem (\ref{dgpb}).
Then, $$\norm{u-u_h}_{L^2(\Omega, \Gamma_1)}+h\snorm{u-u_h}\lesssim h^{\min{(s,p+1)}} \norm{u}_{H^s(\Omega,\Gamma_1)},$$
provided $\gamma$ is chosen sufficiently large.
\end{theorem}
\begin{proof}
By the triangular inequality we have
$$\snorm{u-u_h}\leq \snorm{u-I^h_p u} + \snorm{I^h_p u-u_h}.$$
We first bound the second term on the right-hand side.
Combining the Galerkin orthogonality (\ref{Go}) with the continuity and the coervicity estimates (\ref{continuity_a})-(\ref{coercivity_a}), we obtain:
\begin{align*}\snorm{I^h_p u-u_h}^{2}&\lesssim \mathcal{A}_h(I^h_p u-u_h,I^h_p u-u_h)\\
&= \mathcal{A}_h(I^h_p u-u,I^h_p u-u_h)+\mathcal{A}_h(u-u_h,I^h_p u-u_h)\\
&\lesssim \snorm{I^h_p u-u_h}\snorm{I^h_p u-u}.
\end{align*}
Therefore,
$$\snorm{I^h_p u-u_h}\lesssim \snorm{I^h_p u-u},$$
and
$$\snorm{u-u_h}\lesssim \snorm{u-I^h_p u}.$$
Then, using Lemma \ref{S_IntEstimates}, we get
\begin{equation}\label{energy_error}
\snorm{u-u_h}\lesssim h^{\min{(s-1,p)}} \norm{u}_{H^s(\Omega,\Gamma_1)}.
\end{equation}
For the $L^2$ error estimate, we consider the following adjoint problem: find $\zeta$ such that

\begin{equation*}
\left\{ \begin{array}{lll}
-\Delta \zeta &= u-u_h, &\text{ in }\Omega,\\
\partial_n \zeta &= -\alpha \zeta + \beta \Delta_{\Gamma} \zeta + (u-u_h),&\text{ on }\Gamma_1, \\
\end{array} \right.
\end{equation*}

As $u-u_h\in L^2(\Omega,\Gamma_1)$, using Theorem \ref{th_principe} yields an unique $\zeta\in H^{2}(\Omega,\Gamma_1)$ satisfiying the following stability estimate
$$\norm{\zeta}_{H^2(\Omega,\Gamma_1)}\lesssim \norm{u-u_h}_{L^2(\Omega,\Gamma_1)}.$$
Using Lemma \ref{S_IntEstimates} with $p=1$, we get
\begin{equation}\label{norm_dual}
\snorm{\zeta-I^h_1 \zeta}\lesssim h \norm{\zeta}_{H^2(\Omega,\Gamma_1)}\lesssim h \norm{u-u_h}_{L^{2}(\Omega,\Gamma_1)}.
\end{equation}
Since $\mathcal{A}_h(\cdot,\cdot)$ defined in (\ref{mcalah}) is symmetric, it is easy to see that it holds
\begin{equation}\label{adjpb}
\mathcal{A}_h(\chi, \zeta)=(u-u_h,\chi)_{L^2(\Omega)}+(u-u_h,\chi)_{L^2(\Gamma_1)} \qquad \forall \chi\in H^2(\mathcal{T}_h).
\end{equation}
Next, choosing $\chi=u-u_h$ in (\ref{adjpb}) and employing (\ref{Go}) together with (\ref{continuity_a})
, we find
\begin{eqnarray}
\norm{u-u_h}_{L^2(\Omega, \Gamma_1)}^2 &=& \mathcal{A}_h(u-u_h,\zeta)\nonumber\\
&=&\mathcal{A}_h(u-u_h, \zeta-I^h_1 \zeta)\nonumber\\
&\lesssim& \snorm{u-u_h}\snorm{\zeta-I^h_1 \zeta}.\nonumber
\end{eqnarray}
The thesis follows using (\ref{energy_error}) and (\ref{norm_dual}).
\end{proof}

\section{The parabolic problem and its fully-discretization}\label{s:parabolic}
In this section we employ the results obtained in the previous section to present and analyze a DG space semi-discretization combined with an backward Euler time advancing scheme for solving the following parabolic problem:
\begin{equation}
\left\{\label{mainevopb}\begin{array}{lll}
\partial_t u &= \Delta u + f, &\text{ in }\Omega,\ 0<t\leq T,\\
\partial_n u &= -\alpha u + \beta \Delta_{\Gamma}u - \lambda \partial_t u + g,&\text{ on }\Gamma_1,\ 0<t\leq T,\\
\multicolumn{2}{l}{\text{periodic boundary conditions},}&\text{ on }\Gamma_2,\ 0<t\leq T,\\
u_{|t=0}&=u_0, & \text{ in } \overline{\Omega},
\end{array}\right.
\end{equation}
where $T>0$, $\alpha, \beta, \lambda$ are positive constants and $f,g,u_0$ are (regular enough) given data.
The weak formulation of (\ref{mainevopb}) reads: for any $t\in(0,T]$, find $u
$ such that:
\begin{equation}\label{semiweakevopb}
\begin{cases}
(\partial_t u,v)_{L^2(\Omega)}+\lambda (\partial_t u,v)_{L^2(\Gamma_1)}+ a(u,v)=(f,v)_{L^2(\Omega)} +(g,v)_{L^2(\Gamma_1)},\\
u_{|t=0}=u_0,
\end{cases}
\end{equation}
for any $v \in H^1(\Omega,\Gamma_1)$.\\

It is possible to prove the following result dealing with the existence and (higher) regularity of the weak solution of \eqref{mainevopb}.
\begin{theorem}\label{thm:regularity}
If $u_0\in H^2(\Omega,\Gamma_1)$, $f\in H^1(0,T; L^2(\Omega))$ and $g\in H^1(0,T;L^2(\Gamma_1))$ and the following compatibility conditions 
holds
\begin{enumerate}
\item $u_1:=\Delta u_0 + f (0,\cdot) \in L^2(\Omega)$, 
\item $u_{1\vert_{\Gamma_1}}:=\beta \Delta_\Gamma u_0 -\partial_n u_0 - \alpha u_0 + g(0,\cdot) \in L^2 (\Gamma_1)$,
\end{enumerate}
then problem \eqref{mainevopb} admits a unique solution $u$ with 
$$ u \in C([0,T]; H^2(\Omega,\Gamma_1))\cap C^1([0,T]; L^2_\lambda(\Omega,\Gamma_1)) \cap  
H^1(0,T; H^1_\lambda(\Omega,\Gamma_1)).$$
Moreover, if $u_0\in H_\lambda^{2m} (\Omega;\Gamma_1)$, $\frac{d^k f}{d t^k}  \in H^1(0,T; H^{2m-2k-2}(\Omega))$ and  $\frac{d^k g}{d t^k}  \in H^1(0,T; H^{2m-2k-2}(\Gamma_1))$, for $k=0,\ldots,m-1$ and the following higher order compatibility conditions hold for $k=1,\ldots,m$
\begin{enumerate}
\item[3.] $u^{(k)}_1:=\Delta u_1^{(k-1)} + \frac{d^{k-1}}{d t^{k-1}} f (0,\cdot) \in L^2(\Omega)$
\item[4.] $u_{1\vert_{\Gamma_1}}^{(k)}:=\beta \Delta_\Gamma u_{1\vert_{\Gamma_1}}^{(k-1)} -
\partial_n  u_1^{(k-1)}  -  \alpha u_{1\vert_{\Gamma_1}}^{(k-1)}  + \frac{d^{k-1}}{d t^{k-1}} g(0,\cdot) \in L^2 (\Gamma_1)$,
\end{enumerate}
where we set $u^{(0)}_1:=u_1$ and $u_{1\vert_{\Gamma_1}}^{(0)}= u_{1\vert_{\Gamma_1}}$,  then it holds
for $k=0,\ldots,m-1$
\begin{eqnarray}
\frac{d^k u}{d t^k}  &\in& C([0,T]; H^{2m-2k}(\Omega,\Gamma_1))\cap C^1(0,T; H^{2m-2k-2}_\lambda(\Omega,\Gamma_1))\nonumber\\ 
&&\cap ~ H^1(0,T; H^{2m-2k-1}_\lambda(\Omega,\Gamma_1)). \label{higher}
\end{eqnarray}
\end{theorem}
\begin{proof}
See Appendix A.
\end{proof}

Employing the DG notations introduced in Section~\ref{dgdiscr}, the space semi-discretization of problem (\ref{mainevopb}) becomes: find
{\color{black} $u_h\in C^0(0,T;V^p(\mathcal{T}_h))$} such that, for any $t\in(0,T]$,
\begin{equation}\label{semiweakdgevopb}
\begin{cases}
(\partial_t u_h,v_h)_{L^2(\Omega)}+\lambda(\partial_t u_h,v_h)_{L^2(\Gamma_1)}+ \mathcal{A}_h(u_h,v_h)=(f,v_h)_{L^2(\Omega)} +(g,v_h)_{L^2(\Gamma_1)},\\
{u_h}_{|t=0}=u_{h0},
\end{cases}
\end{equation}
for any $v_h \in V^p(\mathcal{T}_h)$, where $u_{h0}\in V^p(\mathcal{T}_h)$ is the $L^2$-projection of $u_0$ into $V^p(\mathcal{T}_h)$.\\
The following result shows the existence of a unique solution $u_h$ of problem (\ref{semiweakdgevopb}).
\begin{theorem}
The semi-discrete problem (\ref{semiweakdgevopb}) admits a unique local solution.
\end{theorem}
\begin{proof}
As the proof is standard, we only sketch it. Let $\{\phi_j\}_{j= 1}^N$ be an orthogonal basis of $V^p(\mathcal{T}_h)$. The semi-discrete problem (\ref{semiweakdgevopb}) is equivalent to solve, for any $t\in(0,T]$, the following system of ordinary differential equations
\begin{equation}\label{semiweakdgevopbM}
\begin{cases}
(\partial_t u_h,\phi_j)_{L^2(\Omega)}+\lambda(\partial_t u_h,\phi_j)_{L^2(\Gamma_1)}+ \mathcal{A}_h(u_h,\phi_j)=(f,\phi_j)_{L^2(\Omega)} +(g,\phi_j)_{L^2(\Gamma_1)},\\
{u_h}_{|t=0}=u_{h0},
\end{cases}
\end{equation}
for $j=1,...,N$. Setting $u_h=\sum_{i=1}^N c_i(t) \phi_i$, (\ref{semiweakdgevopbM}) can be equivalently written as
\begin{equation}\label{semiweakdgevopbMa}
\begin{cases}
M\dot{\mathbf{c}}(t) + A \mathbf{c}(t) = \mathbf{F}(t),\\
\mathbf{c}(0)=\mathbf{c}^0,
\end{cases}
\end{equation}
where $\mathbf{c}(t)=\left(c_i(t)\right)_{1\leq i \leq N}$, $\mathbf{c}^0=\left(c_i^0\right)_{1\leq i \leq N}$ with $u_{h0}=\sum_{i=1}^N c_i^0 \phi_i,$ and, for $i,j=1,..,N$,
$$A_{ij}=\mathcal{A}(\phi_i,\phi_j), \qquad M_{ij}=(\phi_i,\phi_j)_{L^2(\Omega)} + \lambda(\phi_i,\phi_j)_{L^2(\Gamma_1)},$$
$$F_{i}=(f,\phi_i)_{L^2(\Omega)} + (g,\phi_i)_{L^2(\Gamma_1)}.$$
Since the matrix $M$ is positive definite and $\mathbf{F}(t)\in L^2(0,T;\mathbb{R}^N)$ invoking the well known Picard-Lindel\"of theorem yields the existence and uniqueness of a local solution $\mathbf{c}\in H^1(0,T_N; \mathbb{R})$, i.e. $u_h\in H^1(0,T_N; V^p(\mathcal{T}_h)\subset C([0,T_N];V^p(\mathcal{T}_h))$ with $T_N\in (0,T]$.
\end{proof}
The next result shows the stability of the semi-discrete solution of (\ref{semiweakdgevopb}).

\begin{lemma}
Let $u_h$ be the solution of (\ref{semiweakdgevopb}). Then it holds
\begin{multline}\label{(2)}
\norm{u_h(T)}^2_{L^2_\lambda(\Omega,\Gamma_1)} + \int_0^T\snorm{u_h}^2 dt \lesssim\\ \norm{u_{h0}}^2_{L^2_\lambda(\Omega,\Gamma_1)}+ \int_0^T(\norm{f}_{L^2(\Omega)}^2+\norm{g}_{L^2(\Gamma_1)}^2)dt.
\end{multline}
\end{lemma}

\begin{proof}
Choosing $v_h=u_h$ in (\ref{semiweakdgevopb}) and using (\ref{coercivity_a}) we get
\begin{equation*}
\frac{1}{2}\frac{d}{dt} \norm{u_h}^2_{L^2_\lambda(\Omega,\Gamma_1)} + \snorm{u_h}^2 \lesssim \big(\norm{f}_{L^2(\Omega)}+\norm{g}_{L^2(\Gamma_1)}\big)\norm{u_h}_{L^2_1(\Omega,\Gamma_1)}.
\end{equation*}
Using the arithmetic-geometric inequality and the Poincar\'e-Friedrichs inequality for functions in the broken Sobolev space $H^1(\mathcal{T}_h)$, i.e.,
\begin{equation*}\begin{array}{lll}
\norm{v_h}_{L^2(\Omega)}&\lesssim \big(\seminorm{v_h}_{H^1(\mathcal{T}_h)}^2+\norm{[v_h]}_{L^2(\mathcal{E}_h\cup\mathcal{E}_h^{\Gamma_2})}^2\big)^{1/2} & v_h \in H^1(\mathcal{T}_h)\\
\norm{v_h}_{L^2(\Gamma_1)}&\lesssim \big(\seminorm{v_h}_{H^1(\mathcal{E}_h^{\Gamma_1})}^2+ \norm{[v_h]}_{\mathcal{R}_h^{\Gamma_1}}^2\big)^{1/2} &v_h \in H^1(\mathcal{T}_h)
\end{array}
\end{equation*}
cf. \cite{Piecewise_ineq}, we obtain
\begin{equation}\label{(1)}
\frac{d}{dt} \norm{u_h}^2_{L^2_\lambda(\Omega,\Gamma_1)} + \snorm{u_h}^2 \lesssim \norm{f}_{L^2(\Omega)}^2+\norm{g}_{L^2(\Gamma_1)}^2.
\end{equation}
The thesis follows integrating between 0 and T and noting that\\
 $\norm{u_{h0}}^2_{L^2_\lambda(\Omega,\Gamma_1)}~\lesssim~ \norm{u_0}^2_{L^2_\lambda(\Omega,\Gamma_1)}$ because $u_{h0}$ is the $L^2$-projection of $u_0$ into $V^p(\mathcal{T}_h)$.
\end{proof}

Finally, we consider the fully discretization of problem (\ref{semiweakevopb}) by resorting to the Implicit Euler method with time-step $\Delta t>0$. 
Let $t_k = k \Delta t$, $0\leq k \leq K$, with $K=T/\Delta t$, and denote by $u_h^k, k\geq 0$,the approximation of $u_h(t_k)$.
The fully-discrete problem reads as follows: given $u_h^0=u_{h0}$, find $u_h^{k+1}\in V^p(\mathcal{T}_h),$ $0~<~k~\leq~K-1,$ such that
\begin{align}\label{(EI)}
&\left(\dfrac{u_h^{k+1}-u_h^{k}}{\Delta t},v_h\right)_{L^2(\Omega)}+ \lambda \left(\dfrac{u_h^{k+1}-u_h^{k}}{\Delta t},v_h\right)_{L^2(\Gamma_1)} + \mathcal{A}_h(u_h^{k+1},v_h)\\
&\qquad \qquad \qquad \qquad \qquad \qquad=(f(t_{k+1}),v_h)_{L^2(\Omega)} +(g(t_{k+1}),v_h)_{L^2(\Gamma_1)}\notag
\end{align}
for all $v_h \in V^p(\mathcal{T}_h)$.

\section{Stability and error estimates}
This section is devoted to show that the solution of problem (\ref{(EI)}) converges with optimal rate to the continuous solution of (\ref{mainevopb}). We first prove the following stability result. 

\begin{lemma}\label{data_dep}
Let $f^{k}=f(t_{k})$ and $g^{k}=g(t_{k})$, $k=1,...,K$. Then it holds 
\begin{multline}\label{(3)}
\norm{u_h^{K}}^2_{L^2_\lambda(\Omega,\Gamma_1)} + \Delta t \sum_{k=1}^{K} \snorm{u_h^{k}}^2\\
\lesssim \norm{u_{h0}}^2_{L^2_\lambda(\Omega,\Gamma_1)} + \Delta t \sum_{k=1}^{K} \bigg(\norm{f^k}_{L^2(\Omega)}^2+\norm{g^k}_{L^2(\Gamma_1)}^2\bigg).
\end{multline}
\end{lemma}

\begin{proof}
We choose $v_h=u_h^{k+1}$ in (\ref{(EI)}). Using (\ref{coercivity_a}), the identity
$$( z-y,z)=\frac{1}{2}\norm{z}^2-\frac{1}{2}\norm{y}^2+\frac{1}{2}\norm{z-y}^2,$$
and the Cauchy-Schwarz inequality, we obtain
\begin{multline*}
\norm{u_h^{k+1}}^2_{L^2_\lambda(\Omega,\Gamma_1)}-\norm{u_h^{k}}^2_{L^2_\lambda(\Omega,\Gamma_1)}+\norm{u_h^{k+1}-u_h^{k}}^2_{L^2_\lambda(\Omega,\Gamma_1)} + \Delta t \snorm{u_h^{k+1}}^2\\
\lesssim \Delta t \left ( \norm{f^{k+1}}_{L^2(\Omega)}\norm{u_h^{k+1}}_{L^2(\Omega)}+\norm{g^{k+1}}_{L^2(\Gamma_1)}\norm{u_h^{k+1}}_{L^2(\Gamma_1)}\right ).
\end{multline*}
Employing Young's inequality, Poincar\'e-Friedrichs' inequality and summing over $k$ we get the thesis.
\end{proof}

We next state the main result of this section.

\begin{theorem}
Let $u\in C([0,T]; H^s(\Omega,\Gamma_1))\cap H^1(0,T; L^2_\lambda(\Omega,\Gamma_1))$, $s\geq~2$, be the solution of (\ref{semiweakevopb}) and let $u_h$ be the solution of (\ref{(EI)}). If  $\partial_t u \in L^2(0,T; H^s(\Omega,\Gamma_1))$, $\partial_t^2 u \in L^2(0,T; L^2(\Omega,\Gamma_1))$ and $u_h^0$ satisfies 
\begin{equation}\label{ipo_datum}
\norm{u_0-u_h^0}_{L^2_\lambda(\Omega,\Gamma_1)}\lesssim h^{\min(s,p+1)} \norm{u_0}_{H^s(\mathcal{T}_h)},
\end{equation} 
then
\begin{align*}
\norm{u^K-u_h^K}^2_{L^2_\lambda(\Omega,\Gamma_1)}\lesssim & h^{2\min(s,p+1)} \bigg(\norm{u^K}_{H^s_\lambda(\Omega,\Gamma_1)}^2 +\norm{u_0}^2_{H^s_\lambda(\Omega,\Gamma_1)} \\
&\qquad \qquad \qquad \qquad+ \int_{0}^{T}\norm{\partial_t u(t)}^2_{H^s_\lambda(\Omega,\Gamma_1)}\ dt \bigg)\\
&  + \Delta t^2 \int_{0}^{T}\norm{\partial_t^2 u(t)}^2_{L^2_\lambda(\Omega,\Gamma_1)}\ dt,
\end{align*}
and
\begin{align*}
&\Delta t \sum_{k=1}^{K} \snorm{u^k-u_h^k}^2\lesssim h^{2\min(s-1,p)} \bigg(\Delta t \sum_{k=1}^{K}\norm{u^k}_{H^s_\lambda(\Omega,\Gamma_1)}^2\\
&\qquad \qquad \qquad \qquad \qquad \qquad +h^2\norm{u_0}^2_{H^s_\lambda(\Omega,\Gamma_1)} +h^2\int_{0}^{T}\norm{\partial_t u(t)}^2_{H^s_\lambda(\Omega,\Gamma_1)}\ dt \bigg)\\
&\qquad \qquad\qquad \qquad \qquad+ \Delta t^2 \int_{0}^{T}\norm{\partial_t^2 u(t)}^2_{L^2_\lambda(\Omega,\Gamma_1)}\ dt,
\end{align*}
where $u^k=u(t_k),\ k=1,...,K$.
\end{theorem}

\begin{proof}
We first define the elliptic projection $P:H^2(\Omega,\Gamma_1)\rightarrow V^p(\mathcal{T}_h)$ as
\begin{equation}\label{(5)}
\mathcal{A}_h(Pw-w,v_h)=0 \qquad \forall v_h \in V^p(\mathcal{T}_h),
\end{equation}
where $\mathcal{A}_h(\cdot,\cdot)$ is defined as in (\ref{mcalah}).
We note (see Theorem \ref{StatTh}) that $P$ satisfies the bound
\begin{equation}\label{(7)}
\norm{Pw-w}_{L^2_\lambda(\Omega,\Gamma_1)} + h \snorm{Pw-w}\lesssim h^{\min(s,p+1)}\norm{w}_{H^s_\lambda(\Omega,\Gamma_1)},
\end{equation}
for all $w \in H^s(\Omega, \Gamma_1)$, $s\geq 2$.
We next write $u^k-u_h^k=(u^k-Pu^k)+(Pu^k-u_h^k)$
and start to focus on the second term.  
Considering problem (\ref{semiweakdgevopb}) at time $t_{k+1}$, we easily get
\begin{multline}\label{(8)}
\left(\dfrac{Pu^{k+1}-Pu^{k}}{\Delta t},v_h\right)_{L^2(\Omega)}+ \lambda \left(\dfrac{Pu^{k+1}-Pu^{k}}{\Delta t},v_h\right)_{L^2(\Gamma_1)} + \mathcal{A}_h(Pu^{k+1},v_h)\\
 =(f(t_k),v_h)_{L^2(\Omega)} +(g(t_k),v_h)_{L^2(\Gamma_1)}-( E^{k+1},v_h)_{L^2(\Omega)}-\lambda( E^{k+1},v_h)_{L^2(\Gamma_1)},
\end{multline}
for all $v_h \in V^p(\mathcal{T}_h)$, where
$$E^{k+1}=\partial_t u(t_{k+1}) - \frac{1}{\Delta t}(Pu^{k+1}-Pu^{k}).$$
Subtracting (\ref{(EI)}) from (\ref{(8)}), we get that $e_h^k=Pu^k-u_h^k$ satisfies
\begin{align*}
\left(\frac{e_h^{k+1}-e_h^{k}}{\Delta t},v_h\right)_{L^2(\Omega)}&+ \lambda  \left(\frac{e_h^{k+1}-e_h^{k}}{\Delta t},v_h\right)_{L^2(\Gamma_1)} + \mathcal{A}_h(e_h^{k+1},v_h)\\
&=-( E^{k+1},v_h)_{L^2(\Omega)}-\lambda( E^{k+1},v_h)_{L^2(\Gamma_1)},
\end{align*}
for all $v_h \in V^p(\mathcal{T}_h).$
Then, reasoning as in the proof of Lemma \ref{data_dep} , we obtain
\begin{equation}\label{(10)}
\norm{e_h^{K}}^2_{L^2_\lambda(\Omega,\Gamma_1)} + \Delta t \sum_{k=1}^{K} \snorm{e_h^{k}}^2\lesssim \norm{e_h^0}^2_{L^2_\lambda(\Omega,\Gamma_1)} + \Delta t \sum_{k=1}^{K}\norm{E^{k}}_{L^2_\lambda(\Omega,\Gamma_1)}^2.
\end{equation} 
We bound the first term on the right-hand side of \eqref{(10)} using (\ref{ipo_datum}) and (\ref{(7)}):
\begin{align}\label{(11)}
\norm{e_h^0}_{L^2_\lambda(\Omega,\Gamma_1)}&=\norm{Pu_0-u_{h0}}_{L^2_\lambda(\Omega,\Gamma_1)}\notag \\
&\leq\norm{Pu_0-u_0}_{L^2_\lambda(\Omega,\Gamma_1)}+\norm{u_0-u_{h0}}_{L^2_\lambda(\Omega,\Gamma_1)}\notag \\
&\lesssim h^{\min(s,p+1)} \norm{u_0}_{H^s(\mathcal{T}_h)}.
\end{align}
In order to bound the second term on  the right-hand side of (\ref{(10)}) we observe that it holds:
\begin{align*}
E^{k+1}&=\bigg( \partial_t u(t_{k+1}) - \frac{u^{k+1}-u^k}{\Delta t}\bigg) + \frac{(u^{k+1}-Pu^{k+1})-(u^k-Pu^k)}{\Delta t}\\
&= -\frac{1}{\Delta t}\int_{t_k}^{t_{k+1}}\big(t-t_k\big)\ \partial_t^2 u(t)\ dt + \frac{1}{\Delta t}\int_{t_k}^{t_{k+1}} \partial_t \big(u(t)-Pu(t)\big)\ dt\\
\end{align*}
where we employed Taylor's formula.
Therefore, employing the commutation of the operators $P$ and $\partial t$, we have
\begin{align*}
\norm{E^{k+1}}_{L^2_\lambda(\Omega,\Gamma_1)}^2&\lesssim \frac{1}{\Delta t} \left| \left|\int_{t_k}^{t_{k+1}}\big(t-t_k\big)\ \partial_t^2 u(t)\ dt\right| \right|^2_{L^2_\lambda(\Omega,\Gamma_1)}\\
&+\frac{1}{\Delta t}\left| \left|\int_{t_k}^{t_{k+1}}\big(\partial_t u(t)-P \partial_t u(t)\big) dt\right| \right|^2_{L^2_\lambda(\Omega,\Gamma_1)}.
\end{align*}
Using the Cauchy-Schwarz inequality we get
\begin{align*}
&\left|\left|\int_{t_k}^{t_{k+1}}\big(t-t_k\big)\partial_t^2 u(t)\ dt\right|\right|_{L^2_\lambda(\Omega,\Gamma_1)}\\
&\qquad\leq \left(\int_{t_k}^{t_{k+1}}(t-t_k)^2\ dt \right)^{1/2}\left(\int_{t_k}^{t_{k+1}}\norm{\partial_t^2 u(t)}^2_{L^2_\lambda(\Omega,\Gamma_1)}\ dt \right)^{1/2}\\
&\qquad \lesssim \Delta t^{3/2} \left(\int_{t_k}^{t_{k+1}}\norm{\partial_t^2 u(t)}^2_{L^2_\lambda(\Omega,\Gamma_1)}\ dt\right)^{1/2}.
\end{align*}
Hence, 
$$\frac{1}{\Delta t}\left|\left|\int_{t_k}^{t_{k+1}}\big(t-t_k\big)\partial_t^2 u(t)\ dt\right|\right|_{L^2_\lambda(\Omega,\Gamma_1)}^2\lesssim \Delta t^2 \int_{t_k}^{t_{k+1}}\norm{\partial_t^2 u(t)}^2_{L^2_\lambda(\Omega,\Gamma_1)}\ dt.$$
Employing $\partial_t u \in L^2(0,T; H^s(\mathcal{T}_h))$, $s\geq 2$, and (\ref{(7)}), we obtain
\begin{align*}
&\left|\left|\int_{t_k}^{t_{k+1}}\big(\partial_t u(t)-P \partial_t u(t)\big) dt\right|\right|_{L^2_\lambda(\Omega,\Gamma_1)}\\
&\qquad \leq \left(\int_{t_k}^{t_{k+1}}(1)^2\ dt \right)^{1/2}\left(\int_{t_k}^{t_{k+1}}\norm{\partial_t u(t)-P \partial_t u(t)}^2_{L^2_\lambda(\Omega,\Gamma_1)}\ dt \right)^{1/2}\\
&\qquad \lesssim \Delta t^{1/2}\left(\int_{t_k}^{t_{k+1}}\norm{\partial_t u(t)-P \partial_t u(t)}^2_{L^2_\lambda(\Omega,\Gamma_1)}\ dt \right)^{1/2}\\
&\qquad \lesssim \Delta t^{1/2} h^{\min(s,p+1)} \left(\int_{t_k}^{t_{k+1}}\norm{\partial_t u(t)}^2_{H^s_\lambda(\Omega,\Gamma_1)}\ dt\right)^{1/2}.
\end{align*}
Hence,
\begin{multline}\frac{1}{\Delta t}\left|\left|\int_{t_k}^{t_{k+1}}\big(\partial_t u(t)-P \partial_t u(t)\big) dt\right|\right|_{L^2_\lambda(\Omega,\Gamma_1)}^2\\
\lesssim h^{2\min(s,p+1)} \int_{t_k}^{t_{k+1}}\norm{\partial_t u(t)}^2_{H^s_\lambda(\Omega,\Gamma_1)}\ dt.
\end{multline}
Finally, summing over $k$ we get
\begin{align}\label{(12)}
&\Delta t \sum_{k=1}^{K} \norm{E^{k}}_{L^2_\lambda(\Omega,\Gamma_1)}^2 \\
&\qquad \lesssim \Delta t^2 \int_{0}^{T}\norm{\partial_t^2 u(t)}^2_{L^2_\lambda(\Omega,\Gamma_1)}\ dt + h^{2\min(s,p+1)} \int_{0}^{T}\norm{\partial_t u(t)}^2_{H^s_\lambda(\Omega,\Gamma_1)}\ dt,\notag
\end{align}
which concludes the bound for $e_h^{k}$.
Finally, the thesis follow employing the triangle inequality and the bounds (\ref{(10)})-(\ref{(11)}) together with (\ref{(7)})-(\ref{(12)}).
\end{proof}

\section{Numerical experiments}\label{s:num}

In this section we present some numerical results to validate our theoretical estimates.
In the first two examples (cf Sections \ref{ex1} and \ref{ex2}) we consider a test case with periodic boundary conditions and validate our theoretical error estimates. In the last example (cf Section~\ref{ex3}) we show that our theoretical results seem to hold in the case of more general boundary conditions, provided the exact solution of problem (\ref{mainevopb}) is smooth enough.

\subsection{Example 1}\label{ex1}

We consider problem (\ref{mainevopb}) on $\Omega=(0,1)^2$ and choose $f$ and $g$ such that $u=e^{-10t}(1-\cos(2 \pi x))\cos(4 \pi y)$ is the exact solution.\\
We have tested our scheme on a sequence of uniformly refined structured triangular grids with meshsize $h=\sqrt{2}/2^\ell,\ \ell=2,...,7$. In those sets of numerical experiments we have measured the error $e(T)=u(T)-u_h(T)$ at the final observation time $T=0.001$ in the $\norm{\cdot}_{L^2(\Omega)}$ and $\norm{\cdot}_{L^2(\Gamma_1)}$ norms. We have also measured the quantity $(\Delta t \sum_{k=1}^{K}\snorm{e^k}^2)^{1/2}$, being $e^k=u^k-u_h^k$ .\\
In the first set of experiments we used piecewise linear elements ($p=1$) and the following parameters: $\sigma=10$, $\Delta t=10^{-5}$, $\lambda= 10$, $\beta= 5$ $\alpha= 2$. The computed errors and the corresponding computed convergence rates are reported in Table \ref{tab:tcoscos}.
We have repeated the same set of experiments employing piecewise quadratic elements ($p=2$); the results are reported in Table \ref{tab:tcoscos2}.
From the results shown in Table \ref{tab:tcoscos} and Table \ref{tab:tcoscos2}, it is clear that the expected convergence rates are obtained.

\begin{table}[!h]
\begin{center}
\begin{tabular}{|c|c|c|c|c|c|c|}
\hline
$h$ & $\norm{e(T)}_{L^2(\Omega)}$ & rate & $\norm{e(T)}_{L^2(\Gamma_1)}$ & rate & $(\Delta t \sum_{k=1}^{K}\snorm{e^k}^2)^{1/2}$& rate\\
\hline
$\sqrt{2}/2^2$ & 1.836048e-01 & - & 1.908256e-01 & - & 2.281359e-01& \\
$\sqrt{2}/2^3$ & 5.455936e-02 & 1.75 & 5.035380e-02 & 1.92 & 1.186343e-01& 0.94\\
$\sqrt{2}/2^4$ & 1.451833e-02 & 1.91 & 1.278655e-02 & 1.98 & 5.939199e-02& 1.00\\
$\sqrt{2}/2^5$ & 3.688202e-03 & 1.98 & 3.208881e-03 & 1.99 & 2.962468e-02& 1.00\\
$\sqrt{2}/2^6$ & 9.258142e-04 & 1.99 & 8.028862e-04 & 2.00 & 1.480150e-02& 1.00\\
$\sqrt{2}/2^7$ & 2.316573e-04 & 2.00 & 2.006754e-04 & 2.00 & 7.399580e-03& 1.00\\
\hline
\end{tabular}
   \caption{Example 1. Computed errors, $p=1$, $\sigma=10$, $\Delta t=10^{-5}$, $T=0.001$, $\lambda= 10$, $\beta= 5$, $\alpha= 2$.}
\label{tab:tcoscos}
\end{center}
\end{table}

\begin{table}[!h]
\begin{center}
\begin{tabular}{|c|c|c|c|c|c|c|}
\hline
$h$ & $\norm{e(T)}_{L^2(\Omega)}$ & rate & $\norm{e(T)}_{L^2(\Gamma_1)}$ & rate & $(\Delta t \sum_{k=1}^{K}\snorm{e^k}^2)^{1/2}$& rate\\
\hline
$\sqrt{2}/2^2$ & 2.470397e-02 & - & 1.751588e-02 & - & 5.281897e-02& -\\
$\sqrt{2}/2^3$ & 3.027272e-03 & 3.03 & 2.232268e-03 & 2.97 & 1.405198e-02& 1.91\\
$\sqrt{2}/2^4$ & 3.827204e-04 & 2.98 & 2.822643e-04 & 2.98 & 3.602372e-03& 1.96\\
$\sqrt{2}/2^5$ & 4.797615e-05 & 3.00 & 3.539247e-05 & 3.00 & 9.081101e-04& 1.99\\
$\sqrt{2}/2^6$ &  5.992844e-06 & 3.00 & 4.421683e-06 & 3.00 & 2.276766e-04& 2.00\\
$\sqrt{2}/2^7$ & 7.507474e-07 & 3.00 & 5.632338e-07 & 2.97 & 5.593742e-05& 2.02\\
\hline
\end{tabular}
   \caption{Example 1. Computed errors, $p=2$, $\sigma=10$, $\Delta t=10^{-5}$, $T=0.001$, $\lambda= 10$, $\beta= 5$ $\alpha= 2$.}
\label{tab:tcoscos2}
\end{center}
\end{table}

\subsection{Example 2}\label{ex2}

In the second example, we explore the dependencies of the error on the time-step $\Delta t$.
To this aim, we set $f$ and $g$ as in Section~\ref{ex1}.
In Table \ref{tab:tempo} we report the computed errors and convergence rates obtained with  piecewise linear elements ($p=1$) and the following parameters: $k=7$, $\sigma=10$, $T=0.1$, $\lambda= 10$, $\beta= 5$, $\alpha= 2$, $h=\sqrt{2}/2^{7}$ and vary the time integration step $\Delta t$. The numerical results are in agreement with the theoretical estimate.

\begin{table}[!h]
\begin{center}
\begin{tabular}{|c|c|c|c|c|}
\hline
$\Delta t$ & $\norm{e(T)}_{L^2(\Omega)}$ & rate & $\norm{e(T)}_{L^2(\Gamma_1)}$ & rate\\
\hline
$0.1\times2^{0}$& 2.682138e-02 & - & 8.678953e-02 & -\\
$0.1\times2^{-1}$& 1.487984e-02 & 0.85 & 4.905898e-02 & 0.82\\
$0.1\times2^{-2}$& 7.889826e-03 & 0.92 & 2.630006e-02 & 0.90\\
$0.1\times2^{-3}$& 4.050365e-03 & 0.96 & 1.360794e-02 & 0.95 \\
$0.1\times2^{-4}$& 2.028095e-03 & 1.00 & 6.881036e-03 & 0.98 \\
$0.1\times2^{-5}$& 9.897726e-04 & 1.03 & 3.415646e-03 & 1.01 \\
$0.1\times2^{-6}$& 4.664660e-04 & 1.08 & 1.656678e-03 & 1.04 \\
\hline
\end{tabular}
   \caption{Example 2. Computed errors, $k=7$, $p=1$, $\sigma=10$, $T=0.1$, $\lambda= 10$, $\beta= 5$ $\alpha= 2$.}
\label{tab:tempo}
\end{center}
\end{table}

\subsection{Example 3}\label{ex3}

Finally, we consider problem (\ref{mainevopb}) on $\Omega=(0,1)^2$ with homogeneous Dirichlet boundary conditions applied $\Gamma_2$ and  on $\Gamma_1$.
In this case we choose $f$ and $g$ such that $u=t (1-\cos(2\pi x) ) \cos(\pi y)$ is the exact solution.
In Table \ref{tab:tcos1-cos} we report the computed errors and computed convergence rates at the final time $T=0.1$. Those results have been obtained with piecewise linear elements ($p=1$) and with the following choice of parameters: $\sigma=10$, $\Delta t=0.001$, $\lambda= 10$, $\beta= 5$ $\alpha= 2$.
We have ran the same set of experiments employing piecewise quadratic elements ($p=2$); the computed results are shown in Table \ref{tab:tcos1-cos2}.
The results reported in Table \ref{tab:tcos1-cos} and Table \ref{tab:tcos1-cos2} clearly confirm the theoretical rates of convergence even in the cases of Dirichlet boundary conditions instead of periodic ones, at least whenever the exact solution is sufficiently smooth (see Remark \ref{smoothassumption}).

\begin{table}[!h]
\begin{center}
\begin{tabular}{|c|c|c|c|c|c|c|}
\hline
$h$ & $\norm{e(T)}_{L^2(\Omega)}$ & rate & $\norm{e(T)}_{L^2(\Gamma_1)}$ & rate & $(\Delta t \sum_{k=1}^{K}\snorm{e^k}^2)^{1/2}$& rate\\
\hline
$\sqrt{2}/2^2$ & 9.185918e-03 & - & 1.111234e-02 & - & 1.347859e-01& - \\
$\sqrt{2}/2^3$ & 2.704819e-03 & 1.76 & 2.849404e-03 & 1.96 & 6.413467e-02 & 1.07 \\
$\sqrt{2}/2^4$ & 7.279868e-04 & 1.89 & 7.169369e-04 & 1.99 & 3.155837e-02 & 1.02\\
$\sqrt{2}/2^5$ & 1.875124e-04 & 1.96 & 1.797070e-04 & 2.00 & 1.571196e-02 & 1.01\\
$\sqrt{2}/2^6$ & 4.745622e-05 & 1.98 & 4.501545e-05 & 2.00 & 7.847606e-03 & 1.00\\
$\sqrt{2}/2^7$ & 1.192746e-05 & 1.99 & 1.127502e-05 & 2.00 & 3.922783e-03 & 1.00\\
\hline
\end{tabular}
   \caption{Example 3. Computed errors, $p=1$, $\sigma=10$, $\Delta t=0.001$, $T=0.1$, $\lambda= 10$, $\beta= 5$ $\alpha= 2$.}
\label{tab:tcos1-cos}
\end{center}
\end{table}

\begin{table}[!h]
\begin{center}
\begin{tabular}{|c|c|c|c|c|c|c|}
\hline
$h$ & $\norm{e(T)}_{L^2(\Omega)}$ & rate & $\norm{e(T)}_{L^2(\Gamma_1)}$ & rate & $(\Delta t \sum_{k=1}^{K}\snorm{e^k}^2)^{1/2}$& rate\\
\hline
$\sqrt{2}/2^2$ & 1.239177e-03 & - & 1.607590e-03 & - & 2.589798e-02& - \\
$\sqrt{2}/2^3$ &1.543449e-04 & 3.01 & 2.189412e-04 & 2.88 & 6.771702e-03 & 1.93 \\
$\sqrt{2}/2^4$ &1.911957e-05 & 3.01 & 2.788057e-05 & 2.97 & 1.715537e-03 & 1.98 \\
$\sqrt{2}/2^5$ &2.386211e-06 & 3.00 & 3.496808e-06 & 3.00 & 4.307186e-04 & 1.99 \\
$\sqrt{2}/2^6$ &2.990873e-07 & 3.00 & 4.364171e-07 & 3.00 & 1.079691e-04 & 2.00 \\
$\sqrt{2}/2^7$ & 3.777961e-08 & 2.98  & 5.420558e-08 & 3.01 & 2.621607e-05 & 2.04 \\
\hline
\end{tabular}
   \caption{Example 3. Computed errors, $p=2$, $\sigma=10$, $\Delta t=0.001$, $T=0.1$, $\lambda= 10$, $\beta= 5$ $\alpha= 2$.}
\label{tab:tcos1-cos2}
\end{center}
\end{table}

\appendix
\section{Proof of Theorem \ref{thm:regularity}}
\begin{proof}[Proof of Theorem  \ref{thm:regularity}]
As the proof follows is based on standard arguments (see, e.g., \cite[Chapter 7.1]{Evans}), we only sketch the main steps.\\

\noindent 1. {\em Construction of the discrete space}.
Let $\{e_i\}_{i\geq 1}$ be an orthonormal basis of $L^2(\Omega)$ such that $$\int_{\Omega} \nabla e_i \cdot \nabla z = \lambda_i \int_{\Omega} e_i z \qquad \forall z\in H^1(\Omega),\ i\geq 1,$$
i.e., $\lambda_i$ and $e_i$ are respectively the eigenvalues and eigenfunctions of the weak form of eigenvalue problem $-\Delta e = \lambda e$ with homogeneous Neumann and periodic boundary conditions on $\Gamma_1$ and $\Gamma_2$, respectively.
Reordering $\{e_i\}_{i\geq 1}$ such that $\lambda_1=0$, it is easy to see that there holds
$$\int_{\Omega}\nabla e_i \cdot \nabla e_j = 0,\ \ \ \text{ for }i\neq j\qquad \text{ and }\qquad \int_{\Omega}|\nabla e_i|^2=\lambda_i>0,\ \ \ \text{ for } i>1.$$
Let $V^n=\text{span}\{e_i : i=1,...,n\}$, $n\geq 1$, and let $u_0^n$ be the $L^2(\Omega)$- projection of $u_0$ on $V^n$.
Since the domain is regular, the eigenfunctions $e_i$ belong to $H^2(\Omega)$.\\

\noindent 2. {\em Finite-dimensional approximation of (\ref{semiweakevopb})}. We introduce the following finite dimensional problem: find $u^n\in H^1(0,T;V^n)$ such that, for $t\in(0,T)$,
\begin{equation}\label{discretesemiweakevopb}
\begin{cases}
(\partial_t u^n,z)_{L^2(\Omega)}+\lambda (\partial_t u^n,z)_{L^2(\Gamma_1)}+ a(u^n,z)=(f,z)_{L^2(\Omega)} +(g,z)_{L^2(\Gamma_1)},\\
{u^n}_{|t=0}=u_0^n,
\end{cases}
\end{equation}
for all $z \in V^n$, 
In the sequel we prove that problem (\ref{discretesemiweakevopb}) admits a unique solution in $H^1(0,T; V^n)$. We write $$u^n(t)=\sum_{j=1}^n u_j(t)e_j.$$ The problem (\ref{discretesemiweakevopb}) is equivalent to find $\textbf{u}(t)=(u_1(t), ..., u_n(t))^T\in H^1(0,T;\mathbb{R}^n)$ such that, for each $t\in (0,T)$,
\begin{equation*}
\begin{cases}
M \dot{\mathbf{u}}(t) + A \mathbf{u}(t) = \mathbf{F}(t),\\
\mathbf{u}(0)=(u_{0,1}, ..., u_{0,n})^T,
\end{cases}
\end{equation*}
where, for $i,j=1,..,n$,
$$M_{ij}=M^\Omega + \lambda M^{\Gamma_1}:=\delta_{ij}+ \lambda(e_i,e_j)_{L^2(\Gamma_1)},$$
$$A_{ij}=a(e_i,e_j),\qquad F_{i}=(f,e_i)_{L^2(\Omega)} + (g,e_i)_{L^2(\Gamma_1)}, \qquad u_{0,i}=(u_0, e_i)_{L^2(\Omega)}.$$
Since the matrix $M^{\Gamma_1}$ is semi-positive definite, we see that $M$ is positive definite. In addition, $\mathbf{F}(t)\in L^2(0, T; \mathbb{R}^{n})$ and $A:\mathbb{R}^n\rightarrow \mathbb{R}^n$ is Lipschitz continuous.
Therefore, by standard existence theory of ordinary differential equations, there exists a unique solution $\mathbf{u}(t)$ for a.e. $0\leq t\leq T$.\\

\noindent 3.  {\em Energy estimates}.
Taking $z=u^n$ in (\ref{discretesemiweakevopb}) and using the Cauchy-Schwarz inequality, we obtain
\begin{multline}\label{en_estimate}
\frac{d}{dt}\left( \norm{u^n}_{L^2_\lambda(\Omega, \Gamma_1)}^2\right) + \norm{\nabla u^n}^2_{L^2(\Omega)} + \alpha\norm{u^n}^2_{L^2(\Gamma_1)} + \beta\norm{\nabla_\Gamma u^n}^2_{L^2(\Gamma_1)}\\
\lesssim \norm{u^n}_{L^2_\lambda(\Omega, \Gamma_1)}^2 +\norm{f}^2_{L^2(\Omega)} + \norm{g}^2_{L^2(\Gamma_1)}
\end{multline}
for a.e. $t \in [0,T]$.
Using the differential form of the Gronwall's inequality, data regularity and Lemma \ref{gimischilemma}
we obtain
$$\max_{0\leq t\leq T}\norm{u^n(t)}_{L^2_\lambda(\Omega,\Gamma_1)} \lesssim \norm{u_0}_{L^2_\lambda(\Omega, \Gamma_1)}^2 +\norm{f}^2_{L^2(0,T;L^2(\Omega))} + \norm{g}^2_{L^2(0,T;L^2(\Gamma_1))}\leq C.$$
Integrating (\ref{en_estimate}) in $[0,T]$ and employing the above inequality together with data regularity  and Lemma \ref{gimischilemma} we get
$$\norm{u^n}_{L^2(0,T;H^1_\lambda(\Omega,\Gamma_1))}\lesssim \norm{u_0}_{L^2_\lambda(\Omega, \Gamma_1)}^2 +\norm{f}^2_{L^2(0,T;L^2(\Omega))} + \norm{g}^2_{L^2(0,T;L^2(\Gamma_1))}\leq C.$$
On the other hand, taking $z=\partial_t u^n$ in (\ref{discretesemiweakevopb}), integrating in $t$ and using the Cauchy-Schwarz inequality, we obtain, for every $\tau \in (0,T]$,
\begin{align*}
&\frac{1}{2} \int_0^\tau\norm{\partial_t u^n}^2_{L^2_\lambda(\Omega,\Gamma_1)} + \frac{1}{2}\norm{\nabla u^n(\tau)}^2_{L^2(\Omega)} + \frac{\alpha}{2}\norm{u^n(\tau)}^2_{L^2(\Gamma_1)} + \frac{\beta}{2}\norm{\nabla_\Gamma u^n(\tau)}^2_{L^2(\Gamma_1)}\\
&\qquad\qquad\qquad\ \ \ \ \ \ \leq \frac{1}{2}\norm{\nabla u^n_0}^2_{L^2(\Omega)} + \frac{\alpha}{2}\norm{ u^n_0}^2_{L^2(\Gamma_1)}+\frac{\beta}{2}\norm{\nabla_\Gamma u^n_0}^2_{L^2(\Omega)}\\
&\qquad\qquad\qquad\qquad+ \frac{1}{2} \int_0^\tau \norm{f}^2_{L^2(\Omega)} + \frac{1}{2\lambda} \int_0^\tau \norm{g}^2_{L^2(\Gamma_1)},
\end{align*}
where the right-hand side of the above inequality can be bounded using Lemma \ref{gimischilemma} and data regularity. 

Moreover, differentiating \eqref{discretesemiweakevopb} with respect to $t$ and setting $\tilde u^n:= \partial_t u^n$ we get for any $t\in [0,T]$
\begin{equation}\label{eq:diff-discr}
(\partial_t \tilde u^n,z)_{L^2(\Omega)}+\lambda (\partial_t \tilde u^n,z)_{L^2(\Gamma_1)}+ a(\tilde u^n,z)=(\partial_t f,z)_{L^2(\Omega)} +(\partial_t g,z)_{L^2(\Gamma_1)},\\
\end{equation}
 for all $z\in V^n$. Testing \eqref{eq:diff-discr} with $z=\tilde u^n$, it is easy to show that it holds
 \begin{eqnarray}
 &&\| \partial_t u^n\|^2_{L_\lambda^2(\Omega,\Gamma_1)} + 
 \int_0^t  \| \partial_t u^n (s)\|^2_{H_\lambda^1(\Omega,\Gamma_1)}\, ds 
 \lesssim  
 \int_0^t \|\partial_t f (s) \|^2_{L^2(\Omega)}~ds \nonumber\\ 
 &&\qquad+ \int_0^t \|\partial_t g (s) \|^2_{L^2(\Gamma_1) }\, ds + 
 \| \partial_t u^n (0)\|^2_{L_\lambda^2(\Omega,\Gamma_1)}. 
 \end{eqnarray} 
Taking $t=0$ in \eqref{discretesemiweakevopb}, testing with $z=\partial_t u^n (0)$, integrating by parts and employing the Cauchy-Schwarz inequality once more, we obtain 
 $$ 
 \|\partial_t u^n(0)\|^2_{L^2_\lambda(\Omega,\Gamma_1)} \lesssim \| u^n(0)\|^2_{H^2_\lambda(\Omega,\Gamma_1)} +  \| f(0,\cdot) \|^2_{L^2(\Omega)} + \| g(0,\cdot) \|^2_{L^2(\Gamma_1)},
 $$
 whose right-hand side can be bounded by resorting to compatibility conditions, Lemma \ref{gimischilemma} and data regularity assumptions.
 
 Hence, collecting all the above results, we get 
 $$u^n \in C([0,T]; H_\lambda^1(\Omega,\Gamma_1))\cap C^1(0,T; L^2_\lambda(\Omega,\Gamma_1)) \cap  
H^1(0,T; H^1_\lambda(\Omega,\Gamma_1)).$$
4. {\it Existence of the solution $u$}.  Resorting to subsequences $\{u_{m_l}\}_{l=1}^{\infty}$ of $\{u_m\}_{m=1}^{\infty}$, passing to the limit  for $m\to \infty$ and using standard arguments it is possible to prove that there exists a solution $u$ to problem (\ref{semiweakevopb}) with 
 $$u  \in C([0,T]; H_\lambda^1(\Omega,\Gamma_1))\cap C^1(0,T; L^2_\lambda(\Omega,\Gamma_1)) \cap  
H^1(0,T; H^1_\lambda(\Omega,\Gamma_1)).$$

\noindent 5. {\em Uniqueness of the weak solution}. Let $u_1$ and $u_2$ be two solutions of weak problem (\ref{semiweakevopb}) and set $w=u_1-u_2$. By definition, taking $z=w$, we get from (\ref{semiweakevopb})
$$\frac{d}{dt}\left( \norm{w}_{L^2_\lambda(\Omega, \Gamma_1)}^2\right) + \norm{\nabla w}^2_{L^2(\Omega)} + \alpha\norm{w}^2_{L^2(\Gamma_1)} + \beta\norm{\nabla_\Gamma w}^2_{L^2(\Gamma_1)}=0,$$
that implies $w=0$, or $u_1=u_2$ for a.e. $0\leq t \leq T$.

\noindent 6. {\em Improved regularity}. 
Rewriting (\ref{semiweakevopb}) as
$$a(u,v)=(\tilde{f},v)_{L^2(\Omega)} +(\tilde{g},v)_{L^2(\Gamma_1)},$$
where $\tilde{f}= f - \partial_t u \in L^2(0,T,L^2(\Omega))$ and $\tilde{g}= g - \partial_t u \in L^2(0,T,L^2(\Gamma_1))$. Employing Theorem \ref{th_principe} we get $u(t)\in H^2_\lambda(\Omega,\Gamma_1)$ for a.e. $0\leq t \leq T$.\\

\noindent 6. {\em Higher regularity}. We prove \eqref{higher} by induction. From the above discussion  the result holds true for $m=1$. Assume now the validity of \eqref{higher} for some $m>1$, together with the associated higher order compatibility and regularity conditions. 
Differentiating \eqref{mainevopb} with respect to $t$, it is immediate to verify that $\tilde u=\partial_t u$ verifies
\begin{equation}
\left\{\label{tildeevopb}\begin{array}{lll}
\partial_t \tilde u &= \Delta \tilde u + \tilde f, &\text{ in }\Omega,\ 0<t\leq T,\\
\partial_n \tilde u &= -\alpha \tilde u + \beta \Delta_{\Gamma}\tilde u - \lambda \partial_t \tilde u + \tilde g,&\text{ on }\Gamma_1,\ 0<t\leq T,\\
\multicolumn{2}{l}{\text{periodic boundary conditions},}&\text{ on }\Gamma_2,\ 0<t\leq T,\\
\tilde u_{|t=0}&=\tilde u_0, & \text{ in } \overline{\Omega},
\end{array}\right.
\end{equation}
where $\tilde f=\partial_t f$, $\tilde g=\partial_t g$, $\tilde u_0 = f(0,\cdot) + \Delta u_0$ in $\Omega$ and 
$\tilde u_0|_\Gamma =  \beta \Delta_{\Gamma}u_0  - \partial_n u_0 -\alpha u_0  + g(0,\cdot) $ on $\Gamma$.
Since  the pair $(f,g)$ satisfies the higher order compatibility conditions for $k=1,\ldots,m$ then the pair 
$(\tilde f,\tilde g)$ satisfies the same type of compatibility conditions for $k=1,\ldots,m-1$. Hence, it follows for $k=0,\ldots,m-1$
\begin{eqnarray}
\frac{d^k \tilde u}{d t^k}  &\in& C([0,T]; H^{2m-2k}(\Omega,\Gamma_1))\cap C^1([0,T]; H^{2m-2k-2}_\lambda(\Omega,\Gamma_1))\nonumber\\ 
&&\cap ~ H^1(0,T; H^{2m-2k-1}_\lambda(\Omega,\Gamma_1))
\end{eqnarray}
which immediately implies the validity of \eqref{higher} for $k=0,\ldots,m$.

\end{proof}

The following result has been proof in \cite[Lemmas $4.4$ and $4.5$]{gimischi}.
\begin{lemma}\label{gimischilemma}
Let $z \in Z = \{z\in H^2(\Omega)\ |\ \partial_n z = 0 \text{ on } \Gamma_1\}$. If $z_n$ is the $L^2(\Omega)$-projection of $z$ on $V^n$, then
\begin{equation}\label{boundinitial}
\norm{z_n - z}_{H^1_\lambda(\Omega;\Gamma_1)}\rightarrow 0 \text{ when } n\rightarrow \infty.
\end{equation}
Let $V_{\infty}=\cup_{n=1}^\infty V_n$. Moreover, $Z$ and $V_{\infty}$ are dense in $H^1_\lambda(\Omega;\Gamma_1)$.
\end{lemma}

\end{document}